\RequirePackage{fix-cm}
\documentclass[smallextended]{svjour3}       % onecolumn (second format)
\smartqed  % flush right qed marks, e.g. at end of proof
\usepackage{amssymb}
\usepackage[leqno]{amsmath}
\usepackage{cases}
\usepackage{graphicx}
\usepackage{subfigure}
\usepackage{epstopdf}
\usepackage{epsfig}
\usepackage{float}

\usepackage{mathrsfs}
\usepackage{mathabx}

\usepackage{tabularx}

\usepackage{color}
\usepackage{hyperref}

\begin{document}

\title{Uniform convergence of optimal order of a local discontinuous Galerkin method on a Shishkin mesh under a balanced norm 
%\thanks{This research  was supported by National Natural Science Foundation of China (11771257) and Shandong Provincial Natural Science Foundation (ZR2021MA004).}
}
%\subtitle{Supercloseness of continuous interior penalty methods}

\titlerunning{LDG method on  Shishkin  mesh}        % if too long for running head
 
\author{Xiaoqi Ma   \and
Jin Zhang         \and
        Wenchao Zheng %etc.
}


\institute{Corresponding author:  Jin Zhang \at
              School of Mathematics and Statistics, Shandong Normal University,
Jinan 250014, China\\
              \email{jinzhangalex@sdnu.edu.cn}           %  \\
           \and
           Xiaoqi Ma \at
           School of Mathematics and Statistics, Shandong Normal University,
Jinan 250014, China\\
            \email{xiaoqiMa@hotmail.com}  
}

\date{Received: date / Accepted: date}

\maketitle

\begin{abstract}
This article investigates a local discontinuous Galerkin (LDG) method for one-dimensional and two-dimensional singularly perturbed reaction-diffusion problems on a Shishkin mesh. 
%During this process, Since the natural energy norm is too weak to capture the behavior of the boundary layers appearing in the solutions adequately, 
During this process, due to the inability of the energy norm to fully capture the behavior of the boundary layers appearing in the solutions, a balanced norm is introduced.
 By designing novel numerical fluxes and constructing special interpolations, optimal convergences under the balanced norm are achieved in both 1D and 2D cases. Numerical experiments support the main theoretical conclusions.
%In both cases of 1D and 2D, optimal convergences under the balanced norm are achieved by the novelly designed numerical fluxes and specially constructed interpolations in this paper. 
%This article investigates the local discontinuous Galerkin (LDG) method for one-dimensional and two-dimensional singularly perturbed reaction diffusion problems on the Shishkin grid. Due to the weak natural energy norm, we cannot fully capture the behavior of the boundary layer present in the solution. Therefore, we introduce an equilibrium norm to solve the problem. By designing novel numerical fluxes and constructing special interpolations, optimal convergences under the balanced norm are achieved in both 1D and 2D cases. Numerical experiments support the main theoretical conclusions.
\end{abstract}
%%%%%%%%%%%%%%%%%%%%%%%%%%%%%%%%%%%%%%%%%%%%%%
%
%
%
%%%%%%%%%%%%%%%%%%%%%%%%%%%%%%%%%%%%%%%%%%%%%%
\section{Introduction}
Singular perturbation problems are popularly used in various fields, such as fluid mechanics, energy development and electronic science \cite{R2012-finite,m2021-local,z2013-point}. 
%The typical characteristic of their solutions  is the presence of layers. 
Their typical feature is that one or more layers usually appear in their solutions.
To fully resolve the layers and derive the uniform convergence with respect to  perturbation parameters, layer-adapted meshes were introduced in the 1960s \cite{B1969-Towards} and gradually became an active research field \cite{Ma1Zha2:2023-S,miller1996-fitted}.
%In order to fully analyze the layer and obtain uniform convergence regarding perturbation parameters,  introduced layer adaptive grids in 1860.
%In order to fully analyze the layer and obtain uniform convergence regarding perturbation parameters, layer adaptive grids have been introduced since the 1960s and have been an active research field
%For these problems, it is popular to introduce layer-adapted meshes  to fully resolve layers \cite{Lin2010-Layer}. 
%Among them, the Shishkin meshes  have been widely used and analyzed ,  because it has a very simple structure.
Due to its simple structure, Shishkin meshes  have attracted the attention of many researchers, see \cite{Xie1Zhu2Zho3:2013-U,Zha1Liu2:2017-S,miller1996-fitted}.

% There are two kinds of layer-adapted meshes widely used in the literature, which are Bakhvalov-type mesh and Shishkin-type mesh. Shishkin mesh, as a kind of Shishkin-type mesh, is popular among researchers because of its simplicity, easy construction and stable properties. Therefore, there is a large literature on the singularly perturbed problems on Shishkin mesh, see \cite{d2015-uniformly}\cite{k2010-meshes}\cite{M2021-finite}.
% Currently, there are three ways to deal with this problem. The first way is to use more precise and more refined layer-adapted meshes \cite{Lin2010-Layer}, which are mainly divided into the Bakhvalov-type and the Shishkin-type. The second way is to use more stable numerical methods, such as the interior-penalty discontinuous Galerkin method \cite{F2015-Analysis} and the local discontinuous Galerkin method \cite{Z2014-Uniform}. The third way is a combination of the first and second methods \cite{2009-numerical}\cite{Z2013-Convergence}. For the accuracy of the results and the stability of the method, we adopt the third way to solve the singularly perturbed problems.

%Standard numerical methods, such as finite difference method and finite element method have been well analyzed on 
At present, many numerical methods have been well analyzed on  Shishkin meshes, such as the finite difference method, the finite element method and so on \cite{Styn1Styn2:2018-Convection-diffusion,R2008-Robust}. However, these methods cannot accurately calculate the gradient of the solution. %The gradient here is also very important in many physical applications, as it is related to important physical quantities such as surface friction, temperature and so on. 
To solve this problem, the local discontinuous Galerkin (LDG) method \cite{C1998-local} is a good choice.
%, the local discontinuous Galerkin (LDG) method was first proposed as a local DG method for solving nonlinear, time-dependent convection-diffusion systems. 
As a kind of discontinuous Galerkin (DG) method, the LDG method inherits many advantages of the DG method \cite{C2012-Discontinuous}, such as high-order accuracy, good stability  and flexibility on $hp$ adaptivity. These advantages and its ease of calculating gradients make the LDG method more suitable for handling those solutions with  layers.
%In addition, this method often use $\nabla u$ as an auxiliary variable, which is an important physical quantity. 
Therefore, the LDG method has been applied to different singularly perturbed problems in recent years \cite{c2022-balanced,z2013-point}. 

In this paper, we mainly discuss singularly perturbed reaction-diffusion equations. For these reaction-diffusion problems, the authors in \cite{m2021-convergence} have analyzed uniform convergence  under an energy norm. But the energy norm is too weak to fully capture the behavior of the layers appearing in the solutions. For the sake of addressing this issue, researchers have introduced a new norm--balanced norm, which is more suitable than the energy norm for layers; see \cite{R2012-finite,R2015-convergence}.
%It is worth mentioning that the balanced norm is more suitable for layers than the energy norm; see \cite{R2012-finite,R2015-convergence}.
%However, the standard energy norm is too weak to fully capture the behavior of the layers present in the solution. To address this phenomenon, researchers have introduced a new norm - the equilibrium norm, which is more suitable for layering than the energy norm;
%So Cheng et al. introduced a balanced norm, which is stronger than the energy norm and captures layers more appropriately, see \cite{c2022-balanced}. 
%In \cite{c2021-local}, Cheng studied the singularly perturbed problem with two parameters and proved that the LDG method can achieve suboptimal convergence under the balanced norm.
 In \cite{c2022-balanced}, Cheng et al. derived a uniform convergence of optimal order under the assumption that the smooth component of the solution belongs to the finite element space.
%However, this condition is too strong to hold in general.
But the assumption is too strong and generally cannot be established.
% Cheng studied the uniform convergence of the LDG method under the balanced norm for singularly perturbed problems in 1D and 2D. Under the premise that the smooth part of the solution belongs to the finite element space, a new interpolation is constructed, which is the combination of a Gau{\ss}-Radau projection and a local $L^2$ projection in the one-dimensional case. The optimal error estimation of the balanced norm is proved. In the two-dimensional case, the interpolation is constructed as some form of tensor-product of one-dimensional projections, and the result similar to that in one dimension is proved. Considering that it is difficult to realize the condition that the smooth part of the solution belongs to the finite element space, the author tries to find a new method to solve it.

Our goal is to design a LDG method to avoid the aforementioned drawbacks.
For this purpose,  we design a special numerical flux, and it is the key to our LDG method.  Furthermore, we also introduce a new projection, which consists of a Gau{\ss}-Radau projection and a (weighted) local $L^2$ projection. Based on the new numerical fluxes and projection, it is proven  that the LDG solution is convergent with optimal order under the balanced norm. Finally, the LDG method and its convergence analysis are extended from one-dimensional to two-dimensional.

The framework of the article is arranged as follows. First, we obtain the uniform convergence of the LDG method for a one-dimensional singularly perturbed problem in Section 2. In this progress, a LDG method is introduced and a specific numerical flux is designed.  Then we define a Shishkin mesh and present some projections. Finally, the optimal convergence order is derived in the balanced norm. 
In Section 3, we study the LDG method for a singular perturbation problem in 2D, and the content distribution is the same as in Section 2. 
Finally, a numerical experiment is presented to verify the theoretical results. In the paper, $C$ is a general positive constant which is independent of $\epsilon$ and $N$. Here $\epsilon$ is the perturbation parameter and $N$ is the mesh parameter. Furthermore, let $k\ge 1$ be a fixed positive integer.

%%%%%%%%%%%%%%%%%%%%%%%%%%%%%%%%%%%%%%%%%%%%%

\section{LDG method for 1D case}
We study the following singularly perturbed two-point boundary value problem:
\begin{equation}\label{eq1.1}
-\epsilon u''+b(x)u=f(x),\quad \text{$x\in \Omega=(0,1)$},\quad u(0)=u(1)=0,
\end{equation}
in which $0< \epsilon\ll 1$ is the perturbation parameter, $b(x)\geqslant \beta^2 > 0$  and $f(x)$ are sufficiently smooth functions on $\overline{\Omega}$. Here $\beta$ is some positive constant. The solution of 
\eqref{eq1.1} is typically characterized by boundary layers at $x=0$ and $x=1$ of width $O(\sqrt{\epsilon}\ln(1/\epsilon))$.

%%%%%%%%%%%%%%%%%%%%%%
\subsection{LDG method}
First, we divide $\Omega=(0,1)$ into $0=x_0<x_1<\ldots<x_{N-1}<x_N=1.$
Define  $[x_{i-1}, x_{i}]$ as $I_i$ for $i=1, 2, \ldots,N$, and the step size of the interval $I_i$ as $h_i = x_{i} -x_{i-1}$. Then the discontinuous Sobolev space $V_N$ is denoted as $$V_{N}=\{v\in L^2(\Omega):\text{$v|_{I_i}\in \mathbb{P}_k(I_i), 1\le i\le N$ } \},$$ where $\mathbb{P}_k(I_i)$ is the space of polynomials of degree at most $k\ge 1$ on $I_i$.
%$V_N$ is a subspace of the discontinuous Sobolev space $V$. 
%The space $V=\{v\in L^2(\Omega):\text{$v|_{I_j}\in H^1(I_j), 1\le j\le N$ } \}$.
For $v\in V_N$, we introduce $v^{\pm}_i=\lim\limits_{x\to x^{\pm}_i}v(x)$, and define a jump as: $[[v]]_i=v^{-}_i-v^{+}_i$ for $i=1,2,\ldots,N-1$, $[[v]]_0=-v^{+}_0$ and $[[v]]_N=v^{-}_N$.
%We assume $k\ge 1$ in this paper.

Next, the LDG method is introduced for the problem \eqref{eq1.1}. Firstly, let us rewrite it as the following first-order system,
\begin{equation*}
\begin{aligned}
\epsilon^{-1}q=u',\quad \text{in $\Omega$}, \\
-q'+b(x)u=f(x),\quad \text{in $\Omega$}.
\end{aligned}
\end{equation*}
Let $\chi=(r,v)\in V_N \times V_N$ be any test function. Assume that $\left< \cdot,\cdot\right>_{I_i}$ is the inner product in $L^2(I_i)$.
Find $W=(Q, U)\in V_N \times V_N$ such that 
\begin{align}
&\epsilon^{-1}\left<Q,r\right>_{I_i}+\left<U,r'\right>_{I_i}-\widehat{U}_i r^{-}_i+\widehat{U}_{i-1}r^{+}_{i-1}=0,\label{method--1}\\
&\left<Q,v'\right>_{I_i}+\left<bU,v\right>_{I_i}-\widehat{Q}_iv^{-}_i+\widehat{Q}_{i-1}v^{+}_{i-1}=\left<f,v\right>_{I_i}\label{method--2},
\end{align}
 where the numerical fluxes $\widehat{Q}$ and $\widehat{U}$ are defined by
\begin{equation}\label{A-1}
\begin{aligned}
&\widehat{Q}_i=
\left\{
\begin{aligned}
& Q^{+}_0+\lambda_0U^{+}_0, \quad &&\text{$i=0$},\\
& Q^{+}_i, \quad &&\text{$i=1,2,\ldots,N-1$},\\
& Q^{-}_N-\lambda_NU^{-}_N, \quad &&\text{$i=N$},
\end{aligned}
\right. \\
&\widehat{U}_i=
\left\{
\begin{aligned}
& 0, \quad &&\text{$i=0,N$},\\
&U^{-}_{\frac{3}{4}N}-\lambda_{q}[[Q]]_{\frac{3}{4}N}, \quad&& i=\frac{3}{4}N,\\
& U^{-}_i, \quad &&\text{$i=1,2,\ldots,\frac{3}{4}N-1,\frac{3}{4}N+1,\ldots,N-1$},\\
\end{aligned}
\right.\\
\end{aligned}
\end{equation}
where we choose $\lambda_0=\lambda_N=\epsilon^{\frac{1}{2}}$ and $\lambda_q=\epsilon^{-\frac{1}{2}}$.
\begin{remark}
We design a new numerical flux $\widehat{U}_i$, which is the key to convergence analysis of the LDG method in the balanced norm. Unlike \cite{c2022-balanced}, an additional term involved with  $[[Q]]_{\frac{3}{4}N}$ appears in $\widehat{U}_{\frac{3}{4}N}$. By utilizing this new numerical flux, the  optimal convergence order can be derived without the impossible condition in \cite{c2022-balanced}.
\end{remark}
\begin{remark}
The selection of numerical flux \eqref{A-1} in this paper is a variant of the LDG method, see \cite[Section 4.4.2]{Di1Ern2:2012-M} and \cite{Ngu1Per2:2009-motified} for more details. Due to the numerical flux $\widehat{U}_{\frac{3}{4}N}$ depends on $[[Q]]_{\frac{3}{4}N}$, we cannot use \eqref{method--1} to solve for $Q$ in terms of $U$ element-by-element on $I_{\frac{3}{4}N}\cup I_{\frac{3}{4}N+1}$. 
%the discrete diffusion flux $Q$ cannot be locally represented by $U$ on $I_{\frac{3}{4}N}\cup I_{3N/4+1}$ from \eqref{method--1}.%\eqref{method--1} cannot be used to express locally the discrete diffusion flux $Q$ in terms of $U$ on $I_{\frac{3}{4}N}\cup I_{3N/4+1}$. 
This precludes the local elimination of $Q$ on $I_{\frac{3}{4}N}\cup I_{\frac{3}{4}N+1}$, thereby increasing the computational cost of the approximation method. However, in this case, we can derive an optimal convergence order under the balanced norm.
\end{remark}
Define $\left<w,v\right>=\sum\limits_{i=1}^N\left<w,v\right>_{I_i}$ and 
express the scheme \eqref{method--1} and \eqref{method--2} by a compact form: Find  $W=(Q,U)\in V_N\times V_N$ such that
\begin{equation*}
B(W;\chi)=\left<f,v\right>,\quad \forall \chi=(r,v)\in V_N\times V_N,
\end{equation*}
where
\begin{equation}\label{method:2}
\begin{aligned}
B(W;\chi)=&\epsilon^{-1}\left<Q,r\right>+\left<b U,v\right>+\left<U,r'\right>-\sum_{i=1}^{N-1}U^{-}_i[[r]]_i+\left<Q,v'\right>-\sum_{i=0}^{N-1}Q^{+}_i[[v]]_i\\
&-(Q v)_N^{-}+\sum_{i\in \{0,N\}}\lambda_i [[u]]_i[[v]]_i+\lambda_q[[Q]]_{\frac{3}{4}N}[[r]]_{\frac{3}{4}N}.
\end{aligned}
\end{equation}
%%%%%%%%%%%%%%%%%%%%%%%%%%%%%%%%%%%%%%%%%%%
\subsection{Regularity of the solution}
\begin{lemma}\label{lem1}
(Referene \cite[Lemma 6.2]{miller1996-fitted}) For $i=0,1,2,3,\cdots,k+2$, suppose the solution $u$ of \eqref{eq1.1} allows the decomposition  $u(x)=\bar{u}(x)+u_{\epsilon,1}(x)+u_{\epsilon,2}(x)$ with $x\in \overline{\Omega}$, then
\begin{equation*}\label{solution:1}
|\frac{d^i \bar{u}(x)}{d x^i}|\le C,\quad |\frac{d^i u_{\epsilon,1}(x)}{d x^i}|\le C\epsilon^{-\frac{i}{2}}e^{-\frac{\beta x}{\sqrt{\epsilon}}},\quad |\frac{d^i u_{\epsilon,2}(x)}{d x^i}|\le C\epsilon^{-\frac{i}{2}}e^{-\frac{\beta (1-x)}{\sqrt{\epsilon}}}.
\end{equation*}
 Furthermore, for $i=0,1,2,\cdots,k+1$, we have $q=\bar{q}+q_{\epsilon,1}+q_{\epsilon,2}=\epsilon \frac{d\bar{u}}{dx}+\epsilon \frac{d u_{\epsilon,1}}{dx}+\epsilon \frac{d u_{\epsilon,2}}{dx}$ and
\begin{equation*}\label{solution:2}
|\frac{d^i \bar{q}(x)}{d x^i}|\le C\epsilon,\quad |\frac{d^i q_{\epsilon,1}(x)}{d x^i}|\le C\epsilon^{-\frac{i-1}{2}}e^{-\frac{\beta x}{\sqrt{\epsilon}}},\quad |\frac{d^i q_{\epsilon,2}(x)}{d x^i}|\le C\epsilon^{-\frac{i-1}{2}}e^{-\frac{\beta (1-x)}{\sqrt{\epsilon}}}.
\end{equation*}
\end{lemma}

%%%%%%%%%%%%%%%%%%%%%%%%%
\subsection{Shishkin mesh}
%Layer-adapted meshes were first proposed in \cite{L2003-Layer}. 
Now a Shishkin mesh is introduced. Let $N\ge4$ be a multiple of $4$
and the transition point be 
\begin{align*}
\tau_{t}=\frac{\sigma \sqrt{\epsilon}\ln N}{\beta}\le \frac{1}{4},
\end{align*}
where we take $\sigma \ge k+1$. 
%Let $\varphi(t_j)=4t_j \ln N$ be a piecewise continuous mesh generating function and monotonically increasing from $\varphi(0) = 0$, where $t_j\equiv \frac{j}{N}$ for $j=0,1,\ldots,N$.
Divide  $[\tau_{t}, 1-\tau_{t}]$ into $N/2$  equidistant subintervals and in addition, the $N/4$ intervals are evenly distributed on both  the intervals $[0, \tau_{t}] $ and $[1- \tau_{t}, 1]$.
% Moreover, divide each of the intervals $[0, \tau_{t} ]$ and $[1-\tau_{t}, 1]$ into $N/4$ equidistant subintervals. %Furthermore, there are N/4 subintervals evenly distributed in both the intervals $[0, \tau_{t} ]$ and $[1-\tau_{t}, 1]$.
%We define $\Omega = \Omega^f \cup \Omega^c$
%, where $\Omega^f = [0, \tau_{t} ]\cup[1-\tau_{t}, 1]$ and $\Omega^c = [\tau_{t}, 1-\tau_{t} ]$. Divide $[\tau_{t}, 1-\tau_{t}]$ into $N/2$ equidistant subintervals
%and divide each of the intervals $[0, \tau_{t} ]$ and $[1-\tau_{t}, 1]$ into $N/4$ graded subintervals with
%$\varphi$. 
Therefore, the Shishkin mesh are presented by
\begin{equation}\label{mesh:1}
x_i=
\left\{
\begin{split}
& 4\frac{\sigma \sqrt{\epsilon} }{\beta} t_i\ln N, \quad &&\text{for $i=0, 1, \ldots, \frac{1}{4}N$},\\
&\tau_{t}+2(1-2\tau_{t})(t_i-\frac{1}{4}),  &&\text{for $i=\frac{1}{4}N+1, \ldots, \frac{3}{4}N$},\\
&1-\frac{4\sigma \sqrt{\epsilon} \ln N}{\beta} (1-t_i), \quad &&\text{for $i=\frac{3}{4}N+1, \ldots, N$},
\end{split}
\right.
\end{equation}
in which $t_j\equiv \frac{j}{N}$ for $j=0,1,\ldots,N$.
It is obvious that on the layer-adapted mesh \eqref{mesh:1}, $$C\sqrt{\epsilon}N^{-1}\ln N\le h_j \le CN^{-1}$$
for $1\le j\le N$ hold.
%Let $\psi=e^{-\varphi}$, which is called as mesh characterizing function and will characterize
%the uniform convergence behavior of our numerical method.}

%%%%%%%%%%%%%%%%%%%%%%%%%%%

\subsection{Projections}
%In this section, we combine the local (or weighted local) $L^2$ projection and Gau{\ss}-Radau projection to obtain a new interpolation for later analytical proof.
For the later analytical proof, we design a special projection that is obtained by combining the local (or weighted local) $L^2$ projection and the Gau{\ss}-Radau projection.
\begin{itemize}
\item
\textbf{Local $L^2$ projection $\pi$.} For any $w\in L^2(I_i)$, $\pi w\in \mathbb{P}_k(I_i)$ is denoted as:
\begin{equation*}
\begin{aligned}
\left<\pi w, v\right>_{I_i}=\left<w, v\right>_{I_i}, \quad \forall v\in \mathbb{P}_k(I_i),
\end{aligned}
\end{equation*}
for $I_i=(x_{i-1}, x_i)$, $i=1, 2, \ldots, N$.
\item \textbf{Weighted local $L^{2}$ projection $\pi_{b}$.}
 For any $w \in L^{2}(I_i)$, $\pi_{b}w \in \mathbb{P}_k(I_i)$ such that
\begin{equation*}
\begin{aligned}
\left<b\pi_{b}w, v\right>_{I_i} = \left<bw, v\right>_{I_i}, \quad \forall v\in \mathbb{P}_k(I_i),
\end{aligned}
\end{equation*}
hold for each cell $I_{i} = (x_{i-1}, x_{i})$, $i = 1, 2, \dots, N$, where the weight function $b(x)$ satisfies $b(x)\ge \beta^{2} > 0$ for some constant $\beta$. In the special case of $b = 1$, $\pi_{b}\equiv\pi$.
\item
\textbf{Gau{\ss}-Radau projection $\pi^{\pm}$.} For any $w\in H^1(I_i)$, $\pi^{\pm} w\in \mathbb{P}_k(I_i)$ is defined as:
\begin{equation*}
\begin{aligned}
&\left<\pi^{\pm} w,v\right>_{I_i}=\left<w,v\right>_{I_i},\quad \forall v\in \mathbb{P}_{k-1}(I_i),\\
&(\pi^{+}w)^{+}_{i-1}=w^{+}_{i-1},\quad (\pi^{-}w)^{-}_i=w^{-}_i,
\end{aligned}
\end{equation*}
for $I_i= (x_{i-1}, x_{i})$, $i=1, 2, \ldots, N$.
\end{itemize}
The existence and uniqueness of Gau{\ss}-Radau projection is referred to \cite{A1999-An}. Following \cite[Lemma 5.1]{c2021-local}, we obtain the local stabilities and local approximation properties in Lemma \ref{lem4-111}.
\begin{lemma}\label{lem4-111}
Let $||v||_{I_i}=||v||_{L^2(I_i)}$ and $||v||_{L^{\infty}(I_i)}$ be the usual $L^2$ and $L^{\infty}$ norms on $I_i$, 
\begin{align}
&||\pi w||_{I_i}\le C||w||_{I_i},\nonumber \\
&||\pi_{b} w||_{I_i}\le C||w||_{I_i},\nonumber \\
&||\pi^{-}w||_{I_{i}}\le C\left(||w||_{I_i}+h_i^{1/2}|w_i|\right),\nonumber\\
&||\pi^{+}w||_{I_i}\le C\left(||w||_{I_i}+h_i^{1/2}|w_{i-1}|\right),\nonumber\\
&||\mathbb{R}w||_{L^{\infty}(I_i)}\le C||w||_{L^{\infty}(I_i)},\nonumber\quad \mathbb{R}=\pi, \pi_{b}, \pi^{\pm},\\
&||w-\mathbb{R}w||_{L^{m}(I_i)}\le Ch_i^{k+1}||w^{(k+1)}||_{L^{m}(I_i)},\label{eq3.5}\quad k=2,\infty, \quad \mathbb{R}=\pi, \pi_{b}, \pi^{\pm},\nonumber
\end{align}
where $C > 0$ is a constant independent of $h_{i}$ and $w$.
\end{lemma}

Now we provide the following projections:
\begin{equation*}
\begin{aligned}
&(P^{-}_N u)|_{I_i}=
\left\{
\begin{aligned}
& \pi^{-}u, \quad &&\text{for $i=1,2,\ldots,N/4,3N/4+1,\ldots,N$},\\
& \pi_{b} u,\quad &&\text{for $i=N/4+1,\ldots ,3N/4$},\\
\end{aligned}
\right. \\
&(P^{+}_N q)|_{I_i}=
\left\{
\begin{aligned}
& \pi^{+}q, \quad &&\text{for $i=2, 3, \cdots, N$},\\
& \pi q,\quad &&\text{for $i=1$}.\\
\end{aligned}
\right. 
\end{aligned}
\end{equation*}

Assume that $||v||^2=\sum\limits_{i=1}^N||v||_{I_i}^2$ and $||v||_{I_i}^2=\left<v,v\right>_{I_i}$.
\begin{lemma}\label{lem2}
Let $\epsilon\le CN^{-1}$ and $\sigma \ge k+1$. Then on the layer-adapted mesh \eqref{mesh:1}, there is
\begin{align}
&\qquad||u-P^{-}_N u||_{[0, x_{N/4}]\cup [x_{3N/4}, 1]}\le C\epsilon^{\frac{1}{4}}(N^{-1}\ln N)^{k+1},\label{mesh:2}\\
&\qquad||q-P^{+}_N q||\le C\epsilon^{\frac{3}{4}}(N^{-1}\ln N)^{k+1},\label{mesh:3}\\
&\qquad||u-P^{-}_N u||_{L^{\infty}(I_{i})}\le CN^{-(k+1)},\quad i=N/4+1, \cdots, 3N/4,\label{mesh:4}\\
&\qquad||q-P^{+}_N q||_{L^{\infty}(I_{i})}\le C\epsilon^{\frac{1}{2}}(N^{-1}\ln N)^{k+1},\quad i=1, \cdots, N/4, 3N/4+1, \cdots, N,\label{mesh:5}\\
&\qquad||q-P^{+}_N q||_{L^{\infty}(I_{i})}\le C(\epsilon N^{-(k+1)}+\epsilon^{\frac{1}{2}}N^{-\sigma}),\quad i=N/4+1, \cdots, 3N/4.\label{mesh:6}
\end{align}
\end{lemma}
\begin{proof}
According to the arguments in \cite{c2022-balanced} and Lemma \ref{lem4-111}, this lemma can be derived without any difficulties.
\end{proof}
%\begin{remark}
%In \citep[Lemma2.4]{c2022-balanced}, the authors obtain a more accurate estimation of interpolation error under the premise that the smooth function belongs to the finite element space. Here, we refer to \citep[Lemma2.3]{c2022-balanced}, the interpolation error estimate when the smooth function does not belong to the finite element space.
%\end{remark}

%%%%%%%%%%%%%%%%%%%%%%%%%%%%%%%%
\subsection{Convergence analysis}
Recall that ${w}=(q,u)$ is the solution of eqref{eq1.1}.
According to \eqref{method:2}, we define a natural norm by
\begin{equation}\label{norm:1}
|||w|||^2_{E}=B(w;w)= \epsilon^{-1}||q||^2+||b(x)^{\frac{1}{2}}u||^2+\sum_{j\in\{0,N\}}\lambda_j[[u]]_j^2+\lambda_q[[q]]^2_{\frac{3}{4}N}.
\end{equation} 
Thena more powerful balanced norm is introduced, which is denoted as
\begin{equation}\label{norm:2}
|||w|||^2_{B}= \epsilon^{-3/2}||q||^{2}+||b(x)^{\frac{1}{2}}u||^{2}+\sum_{j\in\{0, N\}}[[u]]_j^{2}+\epsilon^{-1}[[q]]^{2}_{\frac{3}{4}N}.
\end{equation}

\begin{theorem}\label{theorem1}
Let $\epsilon\le CN^{-1}$ and $\sigma \ge k+1$ on the mesh \eqref{mesh:1}. Suppose that ${W}=(Q, U)\in V_N \times V_N$ is the solution of \eqref{method--1} and \eqref{method--2}. And suppose that ${w}=(q,u)$ is the solution of \eqref{eq1.1} meeting Lemma \ref{lem1}. Then define ${e=w-W}$ and 
\begin{align*}
&|||e|||_{E}\le  C\left(\epsilon^{\frac{1}{4}}(N^{-1}\ln N)^{k+1}+\epsilon^{\frac{1}{2}}N^{-k}+N^{-(k+1)}\right),\\
&|||e|||_{B}\le C\left((N^{-1}\ln N)^{k+1}+\epsilon^{\frac{1}{4}}N^{-k}\right),
\end{align*}
where $C>0$ is a constant.
\end{theorem}

\begin{proof}
First, we divide the error $e$ into ${e}=(q-Q,u-U)=\bold{\eta-\chi}$ with 
\begin{equation*}
\begin{aligned}
&\bold{\eta}=(\eta_q,\eta_u)=(q-P^{+}_N q,u-P^{-}_N u),\\
&\bold{\chi}=(\chi_q,\chi_u)=(Q-P^{+}_N q,U-P^{-}_N u)\in V_N \times V_N.
\end{aligned}
\end{equation*}

From Lemma \ref{lem1}, the consistency of numerical fluxes and \eqref{method:2}, we have
%\begin{equation}\label{norm:3}
%B(\chi;\chi)=B(\eta;\chi),\quad \forall \chi=(r,v)\in V_N \times V_N.
%\end{equation}
%Then let $\chi=\chi$ in \eqref{norm:3} and using \eqref{method:2}, we derive
\begin{equation}\label{norm:4}
|||\chi|||_E^2=B(\chi;\chi)=B(\eta;\chi)\equiv \sum_{j=1}^8 S_j.
\end{equation}
Here,
\begin{equation*}
\begin{aligned}
&S_1=\epsilon^{-1}\left<\eta_q,\chi_q\right>,\quad  &&S_2=\left<\eta_u,\chi'_q\right>,\\
&S_3=\left<\eta_q,\chi'_u\right>, \quad &&S_4=\left<b\eta_u,\chi_u\right>,\\
&S_5=-\sum_{i=0}^{N-1}(\eta_q)^{+}_i[[\chi_u]]_i-(\eta_q)^{-}_N(\chi_u)^{-}_N,\quad &&S_6=-\sum_{i=1}^{N-1}(\eta_u)^{-}_i[[\chi_q]]_i,\\
&S_7=\lambda_q[[\eta_q]]_{\frac{3}{4}N}[[\chi_q]]_{\frac{3}{4}N},\quad &&S_8=\sum_{i\in\{0,N\}}\lambda_i[[\eta_u]]_i[[\chi_u]]_i.
\end{aligned}
\end{equation*}
Below, we will estimate them in sequence.
%Then, we will analyze them separately.

From H\"{o}lder inequalities and \eqref{mesh:3}, 
\begin{align*}
S_1\le C(\epsilon^{-\frac{1}{2}}||\eta_q||)(\epsilon^{-\frac{1}{2}}||\chi_q||)\le C\epsilon^{\frac{1}{4}}(N^{-1}\ln N)^{k+1}|||\chi|||_E.
\end{align*}

Then from the orthogonality of the projection $P_{N}^{-}u$ and $P_{N}^{+}q$, we have $S_2=S_3=0$.

For $S_4$, we divide it into the following two parts,
$$S_4=\left<b\eta_u,\chi_u\right>_{[0, x_{N/4}]\cup [x_{3N/4}, 1]}+\left<b\eta_u,\chi_u\right>_{[x_{N/4}, x_{3N/4}]},$$
where from H\"{o}lder inequalities and \eqref{mesh:2}, 
\begin{align*}
\left<b\eta_u,\chi_u\right>_{[0, x_{N/4}]\cup [x_{3N/4}, 1]}
   &\le C||\eta_u||_{[0, x_{N/4}]\cup [x_{3N/4}, 1]}|||\chi|||_E\\
   &\le C\epsilon^{\frac{1}{4}}(N^{-1}\ln N)^{k+1}|||\chi|||_E,
\end{align*}
and the definition of $\pi_{b}$, $\left<b\eta_u,\chi_u\right>_{[x_{N/4}, x_{3N/4}]}=0$.
%and 
%let $b(x_{j-\frac{1}{2}})$ be the value of $b(x)$ at the point $x_{j-\frac{1}{2}}$. According to the Lagrange mean value theorem,  there exist $\mu\in (x_{j-1}, x_{j-\frac{1}{2}})$ to make $$b(x)-b(x_{j-\frac{1}{2}})=b'(\mu)(x-x_{j-\frac{1}{2}})$$ true, then 
%\begin{equation*}
%\begin{aligned}
%\left<b\eta_u,\chi_u\right>_{[x_{N/4}, x_{3N/4}]}&=\sum_{j=N/4+1}^{3N/4}\int_{I_{j}}b\eta_{u}\chi_{u}\mr{d}x\\
%&=\sum_{j=N/4+1}^{3N/4}\int_{I_{j}}(b-b(x_{j-\frac{1}{2}})\eta_{u}\chi_{u}\mr{d}x+\sum_{j=N/4+1}^{3N/4}\int_{I_{j}}b(x_{j-\frac{1}{2}})\eta_{u}\chi_{u}\mr{d}x\\
%&
%
%\end{aligned}
%\end{equation*}
%Therefore, 

From H\"{o}lder inequalities, \eqref{mesh:5} and $\lambda_0=\lambda_N =\epsilon^{\frac{1}{2}}$, 
\begin{align*}
S_5&=(\eta_q)^{+}_0(\chi_u)^{+}_0-(\eta_q)^{-}_N(\chi_u)^{-}_N\\
   &\le C\left(\sum_{j\in\{0, N\}}\lambda_i^{-1}[[\eta_q]]^{2}_i \right)^{1/2}\left(\sum_{j\in\{0, N\}}\lambda_i [[\chi_u]]^{2}_i \right)^{1/2}\\
   &\le C \epsilon^{-\frac{1}{4}}||\eta_q||_{L^{\infty}(\Omega)}|||\chi|||_E\\
   &\le C\epsilon^{\frac{1}{4}}(N^{-1}\ln N)^{k+1}|||\chi|||_E.
\end{align*}

From H\"{o}lder inequalities, the trace inequality, \eqref{mesh:4}, $\lambda_q=\epsilon^{-\frac{1}{2}}$ and $\epsilon \le CN^{-1}$, one has
\begin{align*}
S_6&=|-\sum_{i=N/4+1}^{3N/4-1}(\eta_u)^{-}_i[[\chi_q]]_i-(\eta_u)^{-}_{3N/4}[[\chi_q]]_{\frac{3}{4}N}|\\
     &\le C\sum_{i=N/4+1}^{3N/4-1}||\eta_u||_{L^{\infty}(I_i)}h_i^{\frac{1}{2}}||\chi_q||_{I_i \cup I_{i+1}}+C\lambda_q^{-\frac{1}{2}}||\eta_u||_{L^{\infty}(I_{3N/4})}(\sqrt{\lambda_q}[[\chi_q]]_{\frac{3}{4}N})\\
&\le C\left(\epsilon\sum_{i=N/4+1}^{3N/4-1}h_i^{-1}||\eta_u||^{2}_{L^{\infty}(I_i)}\right)^{\frac{1}{2}}|||\chi|||_{E}+
C\epsilon^{\frac{1}{4}}N^{-(k+1)}|||\chi|||_{E}\\
     &\le C\left(\epsilon^{\frac{1}{2}}N^{-k}+C\epsilon^{\frac{1}{4}}N^{-(k+1)}\right)|||\chi|||_E.
\end{align*}

According to \eqref{mesh:5}, \eqref{mesh:6} and $\lambda_q =\epsilon^{-\frac{1}{2}}$, we obatin
\begin{align*}
S_7\le C \left(\lambda_q[[\eta_q]
]^2_{\frac{3}{4}N}\right)^{\frac{1}{2}} \left(\lambda_q[[\chi_q]]^2_{\frac{3}{4}N}\right)^{\frac{1}{2}}\le C\epsilon^{\frac{1}{4}}(N^{-1}\ln N)^{k+1}|||\chi|||_E.
\end{align*}

From H\"{o}lder inequalities, \eqref{mesh:4}, \eqref{mesh:5} and $\lambda_0 =\lambda_N =\epsilon^{\frac{1}{2}}$,  one obtains
\begin{equation*}
\begin{aligned}
S_8&=\sum_{i\in \{0, N\}}\lambda_i[[\eta_u]]_i[[\chi_u]]_i\\
      &\le C\left(\sum_{i\in \{0, N\}}\lambda_i[[\eta_u]]^{2}_i\right)^{\frac{1}{2}}\left(\sum_{i\in \{0,N\}}\lambda_i[[\chi_u]]^{2}_i\right)^{\frac{1}{2}}\\
      &\le C \epsilon^{\frac{1}{4}}||\eta_u||_{L^{\infty}(\Omega)}|||\chi|||_E\\
      &\le C\epsilon^{\frac{1}{4}}(N^{-1}\ln N)^{k+1}|||\chi|||_E.
\end{aligned}
\end{equation*}

By gathering all the above estimates, we get
\begin{equation*}
\begin{aligned}
B(\eta;\chi)\le C\epsilon^{\frac{1}{4}}(N^{-1}\ln N)^{k+1}|||\chi|||_{E}.
\end{aligned}
\end{equation*}
Using \eqref{norm:4}, there is
\begin{equation*}\label{norm:5}
|||\chi|||^2_{E}= B(\chi;\chi)=B(\eta;\chi)\le C\left(\epsilon^{\frac{1}{4}}(N^{-1}\ln N)^{k+1}+\epsilon^{\frac{1}{2}}N^{-k}\right)|||\chi|||_{E}.
\end{equation*}
Therefore,  it is straightforward to derive
\begin{align}
&|||\chi|||_{E}\le C\epsilon^{\frac{1}{4}}(N^{-1}\ln N)^{k+1}+C\epsilon^{\frac{1}{2}}N^{-k},\label{norm:6}\\
&|||\chi|||_{B}\le C\epsilon^{-\frac{1}{4}}|||\chi|||_E\le C(N^{-1}\ln N)^{k+1}+C\epsilon^{\frac{1}{4}}N^{-k}. \label{norm:7}
\end{align} 

Furthermore, from Lemma \ref{lem2}, \ref{norm:1} and \ref {norm:2}, there are
\begin{align}
&|||\eta|||_{E}\le C\epsilon^{\frac{1}{4}}(N^{-1}\ln N)^{k+1}+CN^{-(k+1)},\label{norm:8}\\
&|||\eta|||_{B}\le C(N^{-1}\ln N)^{k+1}.\label{norm:9}
\end{align} 
Combining \eqref{norm:6}, \eqref{norm:7}, \eqref{norm:8} and \eqref{norm:9}, Theorem \ref{theorem1} can be obtained.
%Through approximation errors in lemma \ref{lem5}, the definition of the energy norm in \eqref{eq4.1},  there is
%\begin{equation}\label{eq4.9}
%\begin{aligned}
%|||\eta|||_{E}=\left[\epsilon^{-1}||\eta_q||^2+||b(x)^{\frac{1}{2}}\eta_u||^2+\sum_{j\in\{0,N\}}\lambda_j[[\eta_u]]_j^2+\lambda_q[[\eta_q]]^2_{\frac{3N}{4}}\right]^{\frac{1}{2}}\le CN^{-(k+1)}.
%\end{aligned}
%\end{equation}
%
% \eqref{eq4.8} and \eqref{eq4.9} yield the Theorem \ref{theorem1}. This completes the proof of Theorem \ref{theorem1}.
\end{proof}

\begin{remark}
For Theorem \ref{theorem1}, when we choose $\epsilon \le CN^{-4}$, 
$$|||e|||_{E}\le C\left(N^{-(k+2)}(\ln N)^{k+1}+N^{-(k+1)}\right),\quad
|||e|||_{B}\le C(N^{-1}\ln N)^{k+1}.$$
Unlike \cite{c2022-balanced}, the optimal convergence order can be proved without the condition that the smooth component of the solution belongs to the finite element space.
%We obtained the optimal convergence order without a smooth part, which belongs to the finite element space in AA.
\end{remark}

%%%%%%%%%%%%%%%%%%%%%%%%%%%%%%%%%%%%
\section{LDG method for two-dimensional case}
We consider a two-dimensional singularly perturbed problem:
\begin{align}\label{eq3.1}
&-\epsilon \Delta u+b(x,y)u=f(x,y),\quad \text{in $ \Omega =(0,1)^2$},\\
&u=0,\quad \text{on $ \partial\Omega$}. \nonumber
\end{align}
where $0 < \epsilon \ll 1$ is a perturbation parameter,  $b(x, y) \geqslant 2\beta ^2 > 0$ and $f(x,y)$ are sufficient smooth functions on $\overline{\Omega}$. 
%And it is well-known that the problem has a unique solution, which is characterized by the presence of  boundary layers of width  $O(\sqrt{\epsilon}\ln(1/\epsilon))$ along the entire boundary $ \partial\Omega$. 
It is widely known that there is a unique solution to this problem, characterized by the presence of a boundary layer with a width of $\mathcal{O}(\sqrt{\epsilon}\ln(1/\epsilon))$ along $\partial\Omega$.
%In this section, we present the LDG method for \eqref{eq3.1}.

\subsection{LDG method}
The two-dimensional Shishkin mesh $\{(x_i
, y_j ), i, j=1, 2, \cdots,N\}$ is constructed by the tensor product of the Shishkin mesh in 1D \eqref{mesh:1} in both vertical and horizontal directions.
Here $x_i$ is presented in \eqref{mesh:1} and $y_j$ is defined in a same way. 

Then for $i,j=1,\cdots,N$, we define $I_{i, x}=(x_{i-1},x_i)$ and $J_{j, y}=(y_{j-1},y_j)$. Let $h_{i, x}=x_i-x_{i-1}$ and $h_{j, y}=y_j-y_{j-1}$. In addition, suppose that
$$
\Omega_N=\{\text{$\kappa_{ij}$: $\kappa_{ij} = I_{i, x}\times J_{j, y}$ for $i, j=1,2,3,\cdots, N$}\}
$$ 
is a rectangle partition of the domain $\Omega$. We define the $k$-degree discontinuous finite element space by: $$R_{N}=\{v\in L^2(\Omega):\text{$v|_{\kappa_{ij}}\in \mathbb{Q}_k(\kappa_{ij})$ }, \kappa_{ij}\in \Omega_N \},$$ where $\mathbb{Q}_k(\kappa_{ij})$ is the space of polynomials on $\kappa_{ij}$, with a maximum degree of $k\ge 1$ in each variable. For $v \in R_N$ and $y \in J_{j, y}$, $j = 1, 2, \ldots ,N$, $v^{\pm}_{i,y}=\lim\limits_{x\to x^{\pm}_i}v(x,y)$ and $v^{\pm}_{x,j}=\lim\limits_{y\to y^{\pm}_j}v(x,y)$ are used to express the traces evaluated from the four directions. Then we define the jumps  as:
\begin{align*}
&&&[[v]]_{i,y}=v^{-}_{i,y}-v^{+}_{i,y}\text{ for $i=1,2,\cdots,N-1$},\; &&[[v]]_{0,y}=-v^{+}_{0,y},\;&[[v]]_{N,y}=v^{-}_{N,y}.\\ 
&&&[[v]]_{x,j}=v^{-}_{x,j}-v^{+}_{x,j}\text{ for $j=1,2,\cdots,N-1$},\; &&[[v]]_{x,0}=-v^{+}_{x,0},\;&[[v]]_{x,N}=v^{-}_{x,N}.
\end{align*}

Below, we shall introduce the LDG method  for \eqref{eq3.1}. Firstly, we rewrite the problem \eqref{eq3.1} as follows:
\begin{equation*}
\begin{aligned}
&\epsilon^{-1}p=u_x,\quad&& \text{in $\Omega$}, \\
&\epsilon^{-1}q=u_y,\quad&& \text{in $\Omega$}, \\
&-p_x-q_y+bu=f,\quad&& \text{in $\Omega$}.
\end{aligned}
\end{equation*}
Suppose that $Z=(v,s,r)\in R_N \times R_N\times R_N$ is any test function, and $(\cdot,\cdot)_{D}$ is the $L^2$ inner product in $D\subset \mathbb{R}^2$. Suppose that $\left< \cdot,\cdot\right>_{I}$ is the inner product in $L^2(I)$ with $I\subset \mathbb{R}$.
Find $T=(U,P,Q)\in R_N \times R_N\times R_N$ such that the  variational formulas hold in $\kappa_{ij}$,
\begin{equation}\label{LDG2:1}
\begin{aligned}
&\epsilon^{-1}(P,s)_{\kappa_{ij}}+(U,s_x)_{\kappa_{ij}}-\left< \widehat{U}_{i,y}, s^{-}_{i,y}\right>_{J_{j, y}}+\left< \widehat{U}_{i-1,y}, s^{+}_{i-1,y}\right>_{J_{j, y}}=0,\\
&\epsilon^{-1}(Q,r)_{\kappa_{ij}}+(U,r_y)_{\kappa_{ij}}-\left<\widehat{U}_{x,j}, r^{-}_{x,j}\right>_{I_{i, x}}+\left<\widehat{U}_{x,j-1}, r^{+}_{x,j-1}\right>_{I_{i, x}}=0,\\
&(P,v_x)_{\kappa_{ij}}-\left<\widehat{P}_{i,y}, v^{-}_{i,y}\right>_{J_{j,y}}+\left<\widehat{P}_{i-1,y}, v^{+}_{i-1,y}\right>_{J_{j,y}}+(Q,v_y)_{\kappa_{ij}}-\left<\widehat{Q}_{x,j}, v^{-}_{x,j}\right>_{I_{i,x}}\\
&+\left<\widehat{Q}_{x,j-1}, v^{+}_{x,j-1}\right>_{I_{i, x}}+(bU,v)_{\kappa_{ij}}=(f,v)_{\kappa_{ij}}.
\end{aligned}
\end{equation}
For $y\in J_{j,y}, j=1,2,\ldots,N$,  the numerical fluxes are defined by
\begin{equation*}
\begin{aligned}
&\widehat{P}_{i,y}=
\left\{
\begin{aligned}
& P^{+}_{0,y}+\lambda_{0,y}U^{+}_{0,y}, \quad &&\text{$i=0$},\\
& P^{+}_{i,y}, \quad &&\text{$i=1,2,\cdots,N-1$},\\
& P^{-}_{N,y}-\lambda_{N,y}U^{-}_{N,y}, \quad &&\text{$i=N$},
\end{aligned}
\right. \label{eq:Bakhvalov mesh-Roos}\\
&\widehat{U}_{i,y}=
\left\{
\begin{aligned}
& 0, \quad &&\text{$i=0,N$},\\
&U_{{\frac{3}{4}N},y}^{-}-\lambda_P[[P]]_{{\frac{3}{4}N},y},\quad &&\text{$i=\frac{3}{4}N$},\\
& U^{-}_{i,y}, \quad &&\text{$i=1,2,\cdots,\frac{3}{4}N-1,\frac{3}{4}N+1,\ldots,N-1$},\\
\end{aligned}
\right.\\
\end{aligned}
\end{equation*}
where we take $\lambda_{0,y}=\lambda_{N,y}=\epsilon^{\frac{1}{2}}$ and $\lambda_P=\epsilon^{-\frac{1}{2}}$. Similarly, when $x\in I_{i,x} (i=1,2,\cdots, N)$, we define $\widehat{U}_{x,j}$ and $\widehat{Q}_{x,j}$ for $j=1,2,\ldots,N$, where $\widehat{U}_{x,\frac{3}{4}N}=U_{x,{\frac{3}{4}N}}^{-}-\lambda_Q[[Q]]_{x,{\frac{3}{4}N}}$ and $\lambda_Q=\epsilon^{-\frac{1}{2}}$.

%denote $(w,v)=\sum_{i=1}^N \sum_{j=1}^N(w,v)_{\kappa_{ij}}$ and 
 Then the scheme \eqref{LDG2:1} by a compact form can be expressed as: Find  $T=(U,P,Q)\in R_N\times R_N\times R_N$ to satisfy
\begin{equation*}
B(T;Z)=(f,v),\quad \forall Z=(v,s,r)\in R_N\times R_N\times R_N,
\end{equation*}
where
\begin{equation}\label{LDG2:2}
\begin{aligned}
B(T;Z)=&(b U,v)+\epsilon^{-1}(P,s)+\epsilon^{-1}(Q,r)\\
&+(U,s_x)-\sum_{j=1}^{N}\sum_{i=1}^{N-1}\left<U^{-}_{i,y},[[s]]_{i,y}\right>_{J_{j,y}}
+(U,r_y)-\sum_{i=1}^{N}\sum_{j=1}^{N-1}\left<U^{-}_{x,j},[[r]]_{x,j}\right>_{I_{i, x}}\\
&+(P,v_x)-\sum_{j=1}^{N}\left \{ \sum_{i=0}^{N-1}\left<P^{+}_{i,y},[[v]]_{i,y}\right>_{J_{j,y}}+\left<P^{-}_{N,y},[[v]]_{N,y}\right>_{J_{j,y}} \right\}\\
&+(Q,v_y)-\sum_{i=1}^{N}\left \{ \sum_{j=0}^{N-1}\left<Q^{+}_{x,j},[[v]]_{x,j}\right>_{I_{i, x}}+\left<Q^{-}_{x,N},[[v]]_{x,N}\right>_{I_{i, x}} \right\}\\
&+\sum_{j=1}^{N}\sum_{i\in \{0,N\}}\left<\lambda_{i,y} [[U]]_{i,y},[[v]]_{i,y}\right>_{J_{j,y}}+\sum_{i=1}^{N}\sum_{j\in \{0,N\}}\left<\lambda_{x,j} [[U]]_{x,j},[[v]]_{x,j}\right>_{I_{i, x}}\\
&+\sum_{j=1}^{N}\left<\lambda_P[[P]]_{\frac{3}{4}N,y},[[s]]_{\frac{3}{4}N,y}\right>_{J_{j,y}}+\sum_{i=1}^{N}\left<\lambda_Q[[Q]]_{x,\frac{3}{4}N},[[r]]_{x,\frac{3}{4}N}\right>_{I_{i, x}}.
\end{aligned}
\end{equation}

%%%%%%%%%%%%%%%%%%%%%%%%%%%%%%%%%%%%%%%%%%%
\subsection{Regularity of the solution}
\begin{lemma}\label{lem3}
(\cite{c2022-balanced,Cl1Gr2:2005-motified}) Suppose that $f$ satisfies sufficient compatibility conditions and $b$ is sufficiently smooth. The solution $u$ of \eqref{eq3.1} admits the decomposition
\begin{align*}
u=S+\sum_{t=1}^4 W_t+\sum_{t=1}^4 Z_t, \quad \text{$(x,y)\in \overline{\Omega}$},
\end{align*}
where $S$ is a smooth component, $W_t$ is a boundary layer component and $Z_t$ is a corner layer component.
Furthermore, for $0 \le i+j \le k + 2$, one has
\begin{equation*}\label{solution2:1}
\begin{aligned}
&|\frac{\partial^{i+j}S}{\partial x^i \partial y^j}|\le C,\\
%C(1+\epsilon^{\frac{k+1-i-j}{2}}),\\
&|\frac{\partial^{i+j}W_1}{\partial x^i \partial y^j}|\le C\epsilon^{-\frac{i}{2}}e^{-\frac{\beta x}{\sqrt{\epsilon}}},\\ 
%(1+\epsilon^{\frac{k+1-j}{2}})
|
&\frac{\partial^{i+j}Z_1}{\partial x^i \partial y^j} |\le C\epsilon^{-\frac{i+j}{2}}e^{-\frac{\beta (x+y)}{\sqrt{\epsilon}}},
\end{aligned}
\end{equation*}
and so on for the remaining terms. 
\end{lemma}

%%%%%%%%%%%%%%%%%%%%%%%%%%%%%%%%%%%%%%%%%%%%

\subsection{Projections}
For the later analytical proof, we construct a new projection by combining the local (or the weighted local) $L^2$ projection and the Gau{\ss}-Radau projection \cite{m2021-local}.
\begin{itemize}
\item
\textbf{Local $L^2$ projection $\Pi$.} For each cell $\kappa_{ij} \in \Omega_N$ and $z\in L^{2}({\Omega})$, $\Pi z\in R_N$ is defined as:
\begin{equation*}
\begin{aligned}
(\Pi z, v)_{\kappa_{ij}}=(z, v)_{\kappa_{ij}}, \quad \forall v\in \mathbb{Q}_k(\kappa_{ij}).
\end{aligned}
\end{equation*}
\item 
\textbf{Weighted local $L^2$ projection $\Pi_{b}$.} Let $w\in C^{1}(\Omega_N)$
and $w\ge w_{0} > 0$ be a general weight function. The weighted local $L^2$ projection $\Pi_{w}$ is defined as: For each cell $\kappa_{ij}\in \Omega_N$ and $z\in L^2(\Omega)$, $\Pi _{w}z\in R_N$ meets
$$(w\Pi_{w}z, v)_{\kappa_{ij}} = (wz, v)_{\kappa_{ij}},\quad \forall v\in \mathbb{Q}_{k}(\kappa_{ij})$$
When $w = 1$, the operator shall reduce to the local $L^2$ projection $\Pi$. In this paper, we take $b(x)$ as a  weight function and $b(x)\ge 2\beta^{2}>0$.
\item
\textbf{Gau{\ss}-Radau projection $\Pi^{\pm}$.} First, we introduce the following notation: For $l$, $m\ge 1$
$$
\mathbb{Q}_{l, m}=\left\{\text{$\sum_{i=0}^{l} \sum_{j=0}^m a_{ij}x^i y^j$:  $a_{ij}\in \mathbb{R}$}\right\}.
$$  
 For any $\kappa_{ij} \in \Omega_N$ and $z\in H^1(\kappa_{ij})$, we define $\Pi^{-}_x z,\Pi^{-}_y z, \Pi^{+}_x z,\Pi^{+}_y z\in \mathbb{Q}_k(\kappa_{ij})$ as:
\begin{equation*}
\begin{aligned}
&\left\{
\begin{aligned}
& (\Pi^{-}_x z,w)_{\kappa_{ij}}=(z,w)_{\kappa_{ij}}, \quad &&\text{$\forall w\in \mathbb{Q}_{k-1,k}$},\\
& \left<(\Pi^{-}_x z)^{-}_{i,y},w\right>_{J_{j,y}}=\left<z^{-}_{i, y},w\right>_{J_{j,y}}, \quad &&\forall w\in \mathbb{P}_{k}(J_{j,y}),
\end{aligned}
\right. \\
&\left\{
\begin{aligned}
& (\Pi^{-}_y z,w)_{\kappa_{ij}}=(z,w)_{\kappa_{ij}}, \quad &&\text{$\forall w\in \mathbb{Q}_{k,k-1}$} ,\\
& \left<(\Pi^{-}_y z)^{-}_{x,j},w\right>_{I_{i,x}}=\left<z^{-}_{x,j},w\right>_{I_{i,x}}, \quad &&\forall w\in \mathbb{P}_{k}(I_{i,x}),
\end{aligned}
\right. \\
&\left\{
\begin{aligned}
& (\Pi^{+}_x z,w)_{\kappa_{ij}}=(z,w)_{\kappa_{ij}}, \quad &&\text{$\forall w\in \mathbb{Q}_{k-1,k}$} ,\\
& \left<(\Pi^{+}_x z)^{+}_{i,y},w\right>_{J_{j,y}}=\left<z^{+}_{i,y},w\right>_{J_{j,y}}, \quad &&\forall w\in \mathbb{P}_{k}(J_{j,y}),
\end{aligned}
\right. \\ 
&\left\{
\begin{aligned}
& (\Pi^{+}_y z,w)_{\kappa_{ij}}=(z,w)_{\kappa_{ij}}, \quad &&\text{$\forall w\in \mathbb{Q}_{k,k-1}$} ,\\
& \left<(\Pi^{+}_y z)^{+}_{x,j},w\right>_{I_{i,x}}=\left<z^{+}_{x,j},w\right>_{I_{i,x}}, \quad &&\forall w\in \mathbb{P}_{k}(I_{i,x}).
\end{aligned}
\right. \\ 
\end{aligned}
\end{equation*}

\end{itemize}
The existence and uniqueness of Gau{\ss}-Radau projection is referred to \cite{A1999-An}, and the corresponding approximate estimation refers to \cite[Lemma 4.3]{m2021-local}.
%\begin{lemma}\label{lem5}
%Referring to \cite{@article {MR3143687,
%    AUTHOR = {Zhu, Huiqing and Zhang, Zhimin},
%     TITLE = {Uniform convergence of the {LDG} method for a singularly
%              perturbed problem with the exponential boundary layer},
%   JOURNAL = {Math. Comp.},
%  FJOURNAL = {Mathematics of Computation},
%    VOLUME = {83},
%      YEAR = {2014},
%    NUMBER = {286},
%     PAGES = {635--663},
%      ISSN = {0025-5718},
%   MRCLASS = {65N30 (65N12)},
%  MRNUMBER = {3143687},
%MRREVIEWER = {Elisabeth Ullmann},
%       DOI = {10.1090/S0025-5718-2013-02736-6},
%       URL = {https://doi.org/10.1090/S0025-5718-2013-02736-6},
%}}. There is a constant C independent of any element and $z$ such that
%\begin{align*}
%&||\pi z||_{I_i}\le C||z||_{I_i},\nonumber \\
%&||\pi^{-}z||_{I_{j}}\le C[||z||_{I_i}+h_i^{1/2}|z_i|],\nonumber\\
%&||\pi^{+}z||_{I_i}\le C[||z||_{I_i}+h_i^{1/2}|z_{i-1}|],\nonumber\\
%&||\mathbb{P}z||_{L^{\infty}(I_i)}\le C||z||_{L^{\infty}(I_i)},\nonumber\quad \mathbb{P}=\pi^{\pm},\\
%&||z-\mathbb{P}z||_{L^{k}(I_i)}\le Ch_i^{m+1}||z^{(m+1)}||_{L^{k}(I_i)},\label{eq3.5}\quad k=2,\infty, \quad \mathbb{P}=\pi^{\pm}.
%\end{align*}
%Here, $||v||_{I_i}=||v||_{L^2(I_i)}$ and $||v||_{L^{\infty}(I_i)}$ are the usual $L^2$ and $L^{\infty}$ norms on $I_i$.
%\end{lemma}

Now we design the following projection:
\begin{equation*}
\begin{aligned}
&(P^{-} u)|_{\kappa_{ij}}=
\left\{
\begin{aligned}
& \Pi^{-}_x u,\quad &&\text{$i=1,\cdots,N/4,3N/4+1,\cdots,N-1,\quad j=N/4+1,\cdots,3N/4$},\\
& \Pi^{-}_y u,\quad &&\text{$i=N/4+1,\cdots,3N/4, \quad j=1,\cdots,N/4,3N/4+1,\cdots,N-1 $},\\
& \Pi_{b} u, \quad &&\text{other else},\\
\end{aligned}
\right. \\
&(P^{+}_x  p)|_{\kappa_{ij}}=
\left\{
\begin{aligned}
& \Pi p,\quad &&\text{$i=1,\quad j=1,\cdots,N$},\\
& \Pi^{+}_x p, \quad &&\text{other else},\\
\end{aligned}
\right.  \\
&(P^{+}_y  q)|_{\kappa_{ij}}=
\left\{
\begin{aligned}
& \Pi q,\quad &&\text{$i=1,\cdots,N, \quad j=1$},\\
& \Pi^{+}_y q, \quad &&\text{other else}.\\
\end{aligned}
\right. 
\end{aligned}
\end{equation*}

Let $||z||^2=\sum\limits_{i=1}^N\sum\limits_{j=1}^N||z||_{\kappa_{ij}}^2$ and $||z||^2_{\kappa_{ij}}=(z,z)_{\kappa_{ij}}$. And we will give the following estimates for $u$ and $p$, and  the error estimates for $q$ can be derived in a similar way.
\begin{lemma}\label{lem4}
(\cite{m2021-local}) Let $\epsilon\le CN^{-1}$ and $\sigma \ge k+1$, then on the two-dimensional Shishkin mesh, one has
\begin{align}
&||u-P^{-} u||_{\Omega/\Omega_{22}}\le C\epsilon^{\frac{1}{4}}(N^{-1}\ln N)^{k+1},\label{solution2:2}\\
&||p-P^{+}_x p||\le C\epsilon^{\frac{3}{4}}(N^{-1}\ln N)^{k+1}+C\epsilon N^{-(k+1)}(\ln N)^{k+\frac{3}{2}},\label{solution2:3}\\
&||u-P^{-} u||_{L^{\infty}(\Omega)}\le C(N^{-1}\ln N)^{k+1},\label{solution2:4}\\
&||p-P^{+}_x p||_{L^{\infty}(\Omega)}\le C\epsilon^{\frac{1}{2}}(N^{-1}\ln N)^{k+1},\label{solution2:5}
\end{align}
where $\Omega_{22}:=[x_{N/4}, x_{3N/4}]\times [y_{N/4}, y_{3N/4}]$.
\end{lemma}
\begin{proof}
According to the arguments in \cite{m2021-local}, the proofs of \eqref{solution2:2}, \eqref{solution2:3}, \eqref{solution2:4} and \eqref{solution2:5} can be obtained easily.
\end{proof}

%%%%%%%%%%%%%%%%%%%%%%%%%%%%%%%%%%%

\subsection{Convergence analysis}
Recall that $t=(u,p,q)$ is thesolution of \eqref{eq3.1}.
From \eqref{LDG2:2}, an natural norm is denoed as:
\begin{align}\label{norm2:1}
|||t|||^2_{E}=B(t;t)&= \epsilon^{-1}||p||^2+\epsilon^{-1}||q||^2+||b(x)^{\frac{1}{2}}u||^2\\ \nonumber
             &+\sum^{N}_{j=1}\sum_{i\in\{0,N\}}(\lambda_{i,y},[[u]]_{i,y}^2)_{J_{j,y}}
+\sum^{N}_{i=1}\sum_{j\in\{0,N\}}(\lambda_{x,j},[[u]]_{x,j}^2)_{I_{i,x}}\\ \nonumber
             &+\sum_{j=1}^N(\lambda_P,[[p]]^2_{\frac{3}{4}N,y})_{J_{j,y}}+\sum_{i=1}^N(\lambda_Q,[[q]]^2_{x, \frac{3}{4}N})_{I_{i,x}}.
\end{align}
Then we denote the balanced norm by:
\begin{align}\label{norm2:2}
|||t|||^2_{B}&= \epsilon^{-3/2}||p||^2+\epsilon^{-3/2}||q||^2+||b(x)^{\frac{1}{2}}u||^2+\sum^{N}_{j=1}\sum_{i\in\{0, N\}}(1, [[u]]_{i, y}^2)_{J_{j,y}}\\ \nonumber
&+\sum^{N}_{i=1}\sum_{j\in\{0,N\}}(1, [[u]]_{x, j}^2)_{I_{i,x}}+\sum_{j=1}^N(\epsilon^{-1}, [[p]]^2_{\frac{3}{4}N, y})_{J_{j,y}}+\sum_{i=1}^N(\epsilon^{-1}, [[q]]^2_{x, \frac{3}{4}N})_{I_{i,x}}.
\end{align}

\begin{theorem}\label{theorem2}
Suppose that $\epsilon \le CN^{-1}$ and $\sigma \ge k+1$ on the layer-adapted mesh $\{(x_i, y_j ), i, j=1, 2,\ldots, N\}$. Let $T=(U,P,Q)\in R_N \times R_N\times R_N$ be the solution of \eqref{LDG2:1}, and $t=(u,p,q)$ is the solution of \eqref{eq3.1}. Moreover, we have 
\begin{align*}
&|||e|||_{E}\le  C\left(\epsilon^{\frac{1}{4}}(N^{-1}\ln N)^{k+1}+\epsilon^{\frac{1}{2}}N^{-k}(\ln N)^{k+1}+\epsilon^{\frac{3}{4}}N^{-k}(\ln N)^{k+\frac{3}{2}}+N^{-(k+1)}\right),\\
&|||e|||_{B}\le C\left((N^{-1}\ln N)^{k+1}+\epsilon^{\frac{1}{4}}N^{-k}(\ln N)^{k+1}+\epsilon^{\frac{1}{2}}N^{-k}(\ln N)^{k+\frac{3}{2}}\right),
\end{align*}
where we define ${e=t-T}$ and $C>0$ is a constant. 
\end{theorem}

\begin{proof}
First, by using projection $P^{-} u$, $P^{+}_x p$ and $P^{+}_y q$,  the error is devided as ${e=t-T}={\eta-\chi}$ with
\begin{equation*}
\begin{aligned}
&\bold{\eta}=(\eta_u,\eta_p,\eta_q)=(u-P^{-} u,p-P^{+}_x p, q-P^{+}_y q),\\
&\bold{\chi}=(\chi_u,\chi_p,\chi_q)=(U-P^{-} u,P-P^{+}_x p, Q-P^{+}_y q)\in R_N \times R_N\times R_N.
\end{aligned}
\end{equation*}

According to Lemma \ref{lem3} and the consistency of numerical fluxes, there is
\begin{equation}\label{norm2:3}
B(\chi;Z)=B(\eta;Z),\quad \forall Z=(u,p,q)\in R_N \times R_N\times R_N.
\end{equation}
Set $Z=\chi$ in \eqref{norm2:3} and use \eqref{LDG2:2}, we derive
\begin{equation}\label{norm2:4}
|||\chi|||_E^2=B(\chi;\chi)=B(\eta; \chi)\equiv \sum_{i=1}^{15}Y_i,
\end{equation}
where
\begin{equation*}\label{eq:VN}
\begin{aligned}
&Y_1=(b\eta_u,\chi_u),\quad Y_2=\epsilon^{-1}(\eta_p,\chi_p),\\
&Y_3=(\eta_u,\chi _{p,x}),\quad Y_4=(\eta_p,\chi_{u,x}),\\
&Y_5=-\sum_{j=1}^{N}\sum_{i=1}^{N-1}\left<(\eta_u)^{-}_{i,y},[[\chi_p]]_{i,y}\right>_{J_{j,y}},\\
&Y_6=-\sum_{j=1}^{N}\left \{ \sum_{i=0}^{N-1}\left<(\eta_p)^{+}_{i,y},[[\chi_u]]_{i,y}\right>_{J_{j,y}}+\left<(\eta_p)^{-}_{N,y},[[\chi_u]]_{N,y}\right>_{J_{j,y}} \right\},\\
&Y_7=\sum_{j=1}^{N}\sum_{i\in \{0,N\}}\left<\lambda_{i,y} [[\eta_u]]_{i,y},[[\chi_u]]_{i,y}\right>_{J_{j,y}},\\
&Y_8=\sum_{j=1}^{N}\left<\lambda_P[[\eta_p]]_{\frac{3}{4}N,y},[[\chi_p]]_{\frac{3}{4}N,y}\right>_{J_{j,y}}.
\end{aligned}
\end{equation*}\label{eq:VN}
Because the proof of $q$ is similar to $p$, we just estimate the part of $p$. Next, we will analyze the above parts in sequence.

By applying H\"{o}lder inequalities, the definition of the weighted local $L^2$ projection, \eqref{solution2:2} and \eqref{solution2:3}, 
\begin{align*}
&Y_1\le C||b^{\frac{1}{2}}\eta_u||_{\Omega/\Omega_{22}}||b^{\frac{1}{2}}\chi_u||\le C\epsilon^{\frac{1}{4}}(N^{-1}\ln N)^{k+1}|||\chi|||_E.\\
&Y_2\le (\epsilon^{-1/2}||\eta_p||)(\epsilon^{-1/2}||\chi_p||)\le C\left(\epsilon^{\frac{1}{4}}(N^{-1}\ln N)^{k+1}+\epsilon^{\frac{1}{2}}N^{-(k+1)}(\ln N)^{k+\frac{3}{2}}\right)|||\chi|||_E.
\end{align*}
Here $\Omega_{22}:=[x_{N/4}, x_{3N/4}]\times [y_{N/4}, y_{3N/4}]$. 

Furthermore, from the orthogonality of the projection $P^{-} u$ and $P^{+}_x  p$, we have $Y_3=Y_4=0$.

Let $\Omega_t=\Omega-\left\{\left([0,x_{N/4}]\times [y_{N/4},y_{3N/4}]\right)\cup \left([x_{3N/4},x_{N-1}]\times [y_{N/4},y_{3N/4}]\right)\right\}$. According to the trace inequality, \eqref{solution2:4}, $\epsilon \le CN^{-1}$ and $\lambda_P=\epsilon^{-\frac{1}{2}}$, we get
\begin{align*}
|Y_5|&\le C\sum_{\kappa_{ij}\in \Omega_t, i\neq 3N/4}h_{j,y}^{\frac{1}{2}}||\eta_u||_{L^{\infty}(\kappa_{ij})}||\chi_p||_{L^{\infty}((I_{i}\cup I_{i+1})\times J_{j,y})}+\sum_{j=1}^{N}\left<(\eta_u)^{-}_{3N/4,y},[[\chi_p]]_{3N/4,y}\right>_{J_{j,y}}\\
%&\le C\sum_{\Omega_c, i\neq 3N/4}h_j^{1/2}||\eta_u||_{L^{\infty}(\kappa_{ij})}h_j^{1/2}(h_i h_j)^{-1/2}||\chi_{p}||_{\kappa_{ij}}+ C(\sum_{j=1}^{N}\lambda_p^{-1} ||(\eta_u)^{-}_{3N/4,y}||_{J_j}^2)^{1/2}(\sum_{j=1}^{N}\lambda_p ||[[\chi_p]]_{3N/4,y}||_{J_j}^2)^{1/2}\\
&\le C\sum_{\kappa_{ij}\in \Omega_t, i\neq 3N/4}h_{j,y}^{\frac{1}{2}}h_{i,x}^{-\frac{1}{2}}||\eta_u||_{L^{\infty}(\kappa_{ij})} ||\chi_p||+C\left(\sum_{j=1}^{N}\lambda_P^{-1} h_{j,y}||\eta_u||_{L^{\infty}(\kappa_{ij})}^2\right)^{\frac{1}{2}}|||\chi|||_E\\
&\le C\left(\epsilon^{\frac{1}{2}}N^{-k}(\ln N)^{k+1}+\epsilon^{\frac{3}{4}}N^{-k}(\ln N)^{k+\frac{3}{2}}+\epsilon^{\frac{1}{4}} (N^{-1}\ln N)^{k+1}\right)|||\chi|||_E.
\end{align*}

According to H\"{o}lder inequalities, \eqref{solution2:5}, \eqref{solution2:4} and the condition $\lambda_{0, y}=\lambda_{N, y}=\epsilon^{\frac{1}{2}}$, 
\begin{equation*}
\begin{aligned}
|Y_6|&=|-\sum_{j=1}^{N}\left(\left<(\eta_{p})_{0, y}^{+}, [[\chi_{u}]]_{0, y}\right>_{J_{j,y}}+\left<(\eta_{p})_{N, y}^{-}, [[\chi_{u}]]_{N, y}\right>_{J_{j,y}}\right)|\\
&\le \left(\sum_{j=1}^{N}\sum_{i\in \{0, N\}}\lambda_{i, y}^{-1}h_{j}||\eta_{p}||^{2}_{L^{\infty}(\Omega)}\right)^{\frac{1}{2}}|||\chi|||_{E}\\
&\le C\left(N\epsilon^{-\frac{1}{2}}N^{-1}\epsilon (N^{-1}\ln N)^{2(k+1)}\right)^{\frac{1}{2}}|||\chi|||_{E}\\
&\le C\epsilon^{\frac{1}{4}}(N^{-1}\ln N)^{k+1}|||\chi|||_E.
\end{aligned}
\end{equation*}
And in a similar way,
\begin{equation*}
\begin{aligned}
Y_7&\le C\left(\sum_{j=1}^{N}\sum_{i\in \{0, N\}}\lambda_{i, y}h_{j}||\eta_{u}||^{2}_{L^{\infty}(\Omega)}\right)^{\frac{1}{2}}|||\chi|||_E\le C\epsilon^{\frac{1}{4}}(N^{-1}\ln N)^{k+1}|||\chi|||_E.
\end{aligned}
\end{equation*}
%Using the Cauchy–Schwarz inequality, \eqref{solution2:4} and $\lambda_{0,y}=\lambda_{N,y}=\sqrt{\epsilon}$ to obtain, one has
%\begin{align*}
%S_7\le C\epsilon^{1/4}N^{-(k+1)}|||\chi|||_E
%%&=\sum_{j=1}^{N}\sum_{i\in \{0,N\}}(\lambda_{i,y} [[\eta_u]]_{i,y},[[\chi_u]]_{i,y})_{J_j}\\
%%&\le C\left\{\sum_{j=1}^{N}\sum_{i\in \{0,N\}}\lambda_{i,y}|| [[\eta_u]]_{i,y}||^2_{J_j}\right\}^{1/2}\left\{\sum_{j=1}^{N}\sum_{i\in \{0,N\}}\lambda_{i,y}|| [[\chi_u]]_{i,y}||^2_{J_j}\right\}^{1/2}\\
%%&\le C\left\{\sum_{j=1}^{N}\sum_{i\in \{0,N\}}\lambda_{i,y}h_j||\eta_u||^2_{L^{\infty}(\Omega_N)}\right\}^{1/2}|||\chi|||_E\\
%
%\end{align*}

From H\"{o}lder inequalities, $\lambda_P =\epsilon^{-\frac{1}{2}}$ and \eqref{solution2:5}, one has
\begin{equation*}
\begin{aligned}
Y_8
%&=\sum_{j=1}^{N}(\lambda_p[[\eta_p]]_{3N/4,y},[[\chi_p]]_{3N/4,y})_{J_j}\\
%&\le C\left\{\sum_{j=1}^{N}\lambda_{p}|| [[\eta_p]]_{3N/4,y}||^2_{J_j}\right\}^{1/2}\left\{\sum_{j=1}^{N}\lambda_{p}|| [[\chi_p]]_{3N/4,y}||^2_{J_j}\right\}^{1/2}\\
%&\le C\left\{\sum_{j=1}^{N}\lambda_{p}h_j||\eta_p||^2_{L^{\infty}(\Omega_N)}\right\}^{1/2}|||\chi|||_E\\
   \le C\epsilon^{\frac{1}{4}}(N^{-1}\ln N)^{k+1}|||\chi|||_E.
\end{aligned}
\end{equation*}
%We gather all the above results and get
%\begin{equation*}
%\begin{aligned}
%B(\eta;\chi)\le C\epsilon^{\frac{1}{4}}(N^{-1}\ln N)^{k+1}|||\chi|||_{E}.
%\end{aligned}
%\end{equation*}
Combining \eqref{norm2:3} and \eqref{norm2:4}, there is
\begin{equation*}
\begin{aligned}
|||\chi|||^2_{E}=B(\eta;\chi)&\le C\left(\epsilon^{\frac{1}{4}}(N^{-1}\ln N)^{k+1}+\epsilon^{\frac{1}{2}}N^{-k}(\ln N)^{k+1}\right.\\
&\left.+C\epsilon^{\frac{3}{4}}N^{-k}(\ln N)^{k+\frac{3}{2}}\right)|||\chi|||_{E}.
\end{aligned}
\end{equation*}
Therefore, it is straightforward to derive
\begin{align}
&|||\chi|||_{E}\le C\left(\epsilon^{\frac{1}{4}}(N^{-1}\ln N)^{k+1}+\epsilon^{\frac{1}{2}}N^{-k}(\ln N)^{k+1}+\epsilon^{\frac{3}{4}}N^{-k}(\ln N)^{k+\frac{3}{2}}\right),\label{norm2:5}\\
&|||\chi|||_{B}\le C\epsilon^{-\frac{1}{4}}|||\chi|||_E\le C\left((N^{-1}\ln N)^{k+1}+\epsilon^{\frac{1}{4}}N^{-k}(\ln N)^{k+1}+\epsilon^{\frac{1}{2}}N^{-k}(\ln N)^{k+\frac{3}{2}}\right).\label{norm2:6}
\end{align} 
From Lemma \ref{lem4}, \ref{norm2:1} and \ref {norm2:2}, 
\begin{align}
&|||\eta|||_{E}\le C\epsilon^{\frac{1}{4}}(N^{-1}\ln N)^{k+1}+C\epsilon^{\frac{1}{2}}N^{-(k+1)}(\ln N)^{k+\frac{3}{2}}+CN^{-(k+1)},\label{norm2:7}\\
&|||\eta|||_{B}\le C\epsilon^{\frac{1}{4}}N^{-(k+1)}(\ln N)^{k+\frac{3}{2}}+C(N^{-1}\ln N)^{k+1}.\label{norm2:8}
\end{align} 
From \eqref{norm2:5}, \eqref{norm2:6}, \eqref{norm2:7} and \eqref{norm2:8}, The proof of Theorem \ref{theorem2} is complete.
\end{proof}
\begin{remark}
For Theorem \ref{theorem2}, if we choose $\epsilon \le CN^{-4}$, then 
$$|||e|||_{E}\le CN^{-(k+2)}(\ln N)^{k+1}+CN^{-(k+1)},\quad
|||e|||_{B}\le C(N^{-1}\ln N)^{k+1}.$$
Unlike \cite{c2022-balanced}, the optimal convergence order can be derived when the smooth component of the solution does not belong to the finite element space.
\end{remark}

%%%%%%%%%%%%%%%%%%%%%%%%%%%%%%%%%%%%%%%%%%%%%
%
%
%
%%%%%%%%%%%%%%%%%%%%%%%%%%%%%%%%%%%%%%%%%%%%%

\section{Numerical experiments}
In this section, we provide a numerical example to verify Theorem \ref{theorem1}.
The following singularly perturbed reaction-diffusion problem is considered:
\begin{equation*}
-\epsilon u''(x)+ u(x)=f(x)\quad \text{in $\Omega:=(0,1)$},\quad
u(0)=u(1)=0.
\end{equation*}
Here $f(x)$ is selected to make the solution of the above problem satisfy
\begin{equation*}
u(x)=\frac{1-e^{-1/\sqrt{\epsilon}}}{e^{-x/\sqrt{\epsilon}}-e^{-(1-x)/\sqrt{\epsilon}}}-\cos(\pi x).
\end{equation*}

In this experiment, we choose $\epsilon=10^{-4},10^{-6},10^{-8},10^{-10},10^{-12}$, $N=32,64,\ldots,1024$ and set $k=1,2,3$. On the layer-adapted mesh \eqref{mesh:1}, we choose $\sigma=k+1$.
 Set $e_N$ as the error of $|||e|||_E$ or $|||e|||_B$, where $N$ is the mesh parameter. 
The convergence rates are calculated with the formula
\begin{equation*}
r_p=\frac{\log e_N-\log e_{2N}}{\log (2\ln N/\ln 2N)}.
\end{equation*}

Table \ref{table-W-1} shows the errors and convergence rates in the energy norm. The convergence rates and errors under the balanced norm are displayed in Table \ref{table-W-2}.
Then the data provided by Table \ref{table-W-1} and Table \ref{table-W-2} support our main conclusion. Furthermore, by comparing the data in Table \ref{table-W-1} and Table \ref{table-W-2} for each $\epsilon$, it is obvious that the results are approximately $|||e|||_B=\epsilon^{-1/4}|||e|||_E$.
%By comparing BB and CC at each KK, the result is approximately AA.
%{\color{red}Through observation, we find that the error result of $|||e|||_E$ and $|||e|||_B$ are mainly controlled by $\epsilon^{1/4}(N^{-1}\ln N)^{k+1}$ and $(N^{-1}\ln N)^{k+1}$, respectively. }
The example in 2D is not provided because it is similar to the situation in 1D.
%Due to the similarity between two-dimensional numerical examples and one-dimensional numerical examples, no two-dimensional numerical examples are provided.
%From Theorem \ref{theorem1}, we can obtain that $|||e|||_B$ is convergent uniformly with an order of $k+1$, and the convergence rate of $|||e|||_B$ depends not only on $\epsilon^{1/2}N^{-k}$ but also on $(N^{-1}\ln N)^{k+1}$. In order to determine which term controls the convergence rate of $|||e|||_B$, we make Table \ref{table-W-1} and Table \ref{table-W-2}. Table \ref{table-W-1} shows the error estimate of $|||e|||_B$ and the convergence rate of the error with respect to $\epsilon^{1/2}N^{-k}$ term, and Table \ref{table-W-2} shows the error estimate of $|||e|||_B$ and the convergence rate of the error with respect to $(N^{-1}\ln N)^{k+1}$ term. It can be obtained from the two Tables that when $\epsilon^{} \le CN^{-4}$, the convergence rate of $|||e|||_B$ is mainly determined by $(N^{-1}\ln N)^{k+1}$ term; when $CN^{-4}\le \epsilon^{} \le CN^{-1}$, the convergence rate of $|||e|||_B$ is mainly determined by $\epsilon^{1/2}N^{-k}$ term. Therefore, we prove that Theorem \ref{theorem1} is sharp. Since the two-dimensional case is similar to the one-dimensional case, the two-dimensional numerical experiments are not presented.

\begin{table}[H] 
\caption{The errors and convergence rates under the energy norm}\label{table-W-1}
\footnotesize
\begin{tabular*}{\textwidth}{@{\extracolsep{\fill}} cccccccccccc }
%\cline{1-13}{}
%    &\multicolumn{12}{c}{$\epsilon_1$ }\\
\cline{1-12}
           &    &\multicolumn{2}{c}{$\epsilon=10^{-4}$} &\multicolumn{2}{c}{$\epsilon=10^{-6}$}  &\multicolumn{2}{c}{$\epsilon=10^{-8}$}   
&\multicolumn{2}{c}{$\epsilon=10^{-10}$} &\multicolumn{2}{c}{$\epsilon=10^{-12}$}   \\
% $N$  &$10^{-7}$  &$10^{-7}$  &$10^{-8}$  &$10^{-8}$  &$10^{-9}$  &$10^{-9}$  &$10^{-10}$  &$10^{-10}$\\
 
\cline{3-12}
$k$ &$N$&$|||e|||_E$&$r_p$&$|||e|||_E$&$r_p$&$|||e|||_E$&$r_p$&$|||e|||_E$&$r_p$&$|||e|||_E$&$r_p$\\
\cline{1-12}

$1$       & $32$       & 0.26E-1  & 1.59    & 0.82E-1  & 1.59  & 0.26E-2   & 1.59  & 0.82E-3   & 1.59  &0.26E-3  &1.59\\
$ $       &$64$        & 0.12E-1  & 1.75    & 0.37E-2  & 1.75  & 0.12E-2   & 1.75  & 0.37E-3   & 1.75  &0.12E-3  &1.75\\
$ $       &$128$       & 0.45E-2  & 1.85    & 0.14E-2  & 1.85  & 0.45E-3   & 1.85  & 0.14E-3   & 1.85  &0.45E-4  &1.85\\
$ $       &$256$       & 0.16E-2  & 1.92    & 0.51E-3  & 1.92  & 0.16E-3   & 1.92  & 0.51E-4   & 1.92  &0.16E-4  &1.92\\
$ $       &$512$       & 0.53E-3  & 1.95    & 0.17E-3  & 1.95  & 0.53E-4   & 1.95  & 0.17E-4   & 1.95  &0.53E-5  &1.95\\
$ $       &$1024$      & 0.17E-3  & ---     & 0.54E-4  & ---   & 0.17E-4   & ---   & 0.54E-5    & ---  &0.17E-5  &---\\

$2$       & $32$       & 0.63E-2  & 2.44    & 0.20E-2  & 2.44  & 0.63E-3   & 2.44  & 0.20E-3   & 2.44  &0.63E-4  &2.44\\
$ $       &$64$        & 0.18E-2  & 2.65    & 0.57E-3  & 2.65  & 0.18E-3   & 2.65  & 0.57E-4   & 2.65  &0.18E-4  &2.65\\
$ $       &$128$       & 0.43E-3  & 2.80    & 0.14E-3  & 2.80  & 0.43E-4   & 2.80  & 0.14E-4   & 2.80  &0.43E-5  &2.80\\
$ $       &$256$       & 0.91E-4  & 2.88    & 0.29E-4  & 2.88  & 0.91E-5   & 2.88  & 0.29E-5   & 2.88  &0.91E-6  &2.88\\
$ $       &$512$       & 0.17E-4  & 2.94    & 0.55E-5  & 2.94  & 0.17E-5   & 2.94  & 0.55E-6   & 2.94 &0.17E-6  &2.94\\
$ $       &$1024$      & 0.31E-5  & ---     & 0.97E-6  & ---   & 0.31E-6   & ---   & 0.97E-7   & ---  &0.31E-7  &---\\

$3$       & $32$       & 0.15E-2  & 3.30    & 0.48E-3  & 3.30  & 0.15E-3   & 3.30  & 0.48E-4   & 3.30 
& 0.15E-4   & 3.30\\
$ $       &$64$        & 0.28E-3  & 3.56    & 0.89E-4  & 3.56  & 0.28E-4   & 3.56  & 0.89E-5   & 3.56
& 0.28E-5   & 3.56\\
$ $       &$128$       & 0.41E-4  & 3.74    & 0.13E-4  & 3.74  & 0.41E-5   & 3.74  & 0.13E-5   & 3.74 &0.41E-6  & 3.74\\
$ $       &$256$       & 0.51E-5  & 3.86    & 0.16E-5  & 3.86  & 0.51E-6   & 3.86  & 0.16E-6   & 3.86   &0.51E-7  & 3.86\\
$ $       &$512$       & 0.56E-6  & ---     & 0.18E-6  & ---   & 0.56E-7   & ---   & 0.18E-7   &---  &0.56E-8  &---\\
%$ $       &$1024$      & ---  & ---    & ---  & ---  & ---   & ---  & ---   & ---  &---  &---\\

\cline{1-12}
\end{tabular*}
\label{table:2}
\end{table}

\begin{table}[H] 
\caption{The errors and convergence rates under the balanced norm}\label{table-W-2}
\footnotesize
\begin{tabular*}{\textwidth}{@{\extracolsep{\fill}} cccccccccccc }
%\cline{1-13}{}
%    &\multicolumn{12}{c}{$\epsilon_1$ }\\
\cline{1-12}
            &    &\multicolumn{2}{c}{$\epsilon=10^{-4}$} &\multicolumn{2}{c}{$\epsilon=10^{-6}$}  &\multicolumn{2}{c}{$\epsilon=10^{-8}$}   
&\multicolumn{2}{c}{$\epsilon=10^{-10}$} &\multicolumn{2}{c}{$\epsilon=10^{-12}$}   \\
% $N$  &$10^{-7}$  &$10^{-7}$  &$10^{-8}$  &$10^{-8}$  &$10^{-9}$  &$10^{-9}$  &$10^{-10}$  &$10^{-10}$\\
 
\cline{3-12}
$k$ &$N$&$|||e|||_B$&$r_p$&$|||e|||_B$&$r_p$&$|||e|||_B$&$r_p$&$|||e|||_B$&$r_p$&$|||e|||_B$&$r_p$\\
\cline{1-12}

$1$       & $32$       & 0.25E+0  & 1.56    & 0.25E+0  & 1.56  & 0.25E+0   & 1.56  & 0.25E+0   & 1.56  &0.25E+0  &1.56\\
$ $       &$64$        & 0.11E+0  & 1.72    & 0.11E+0  & 1.72  & 0.11E+0   & 1.72  & 0.11E+0   & 1.72  &0.11E+0  &1.72\\
$ $       &$128$       & 0.44E-1  & 1.83    & 0.44E-1  & 1.83  & 0.44E-1   & 1.83  & 0.44E-1   & 1.83  &0.44E-1  &1.83\\
$ $       &$256$       & 0.16E-1  & 1.90    & 0.16E-1  & 1.90  & 0.16E-1   & 1.90  & 0.16E-1   & 1.90  &0.16E-1  &1.90\\
$ $       &$512$       & 0.53E-2  & 1.95    & 0.53E-2  & 1.95  & 0.53E-2   & 1.95  & 0.53E-2   & 1.95  &0.53E-2  &1.95\\
$ $       &$1024$      & 0.17E-2  & ---    & 0.17E-2  & ---  & 0.17E-2   & ---  & 0.17E-2   & ---  &0.17E-2  &---\\

$2$       & $32$       & 0.61E-1  & 2.42    & 0.61E-1  & 2.42  & 0.61E-1   & 2.42  & 0.61E-1   & 2.42 &0.61E-1  &2.42\\
$ $       &$64$        & 0.18E-1  & 2.64    & 0.18E-1  & 2.64  & 0.18E-1   & 2.64  & 0.18E-1   & 2.64  &0.18E-1  &2.64\\
$ $       &$128$       & 0.43E-2  & 2.79    & 0.43E-2  & 2.79  & 0.43E-2   & 2.79  & 0.43E-2   & 2.79  &0.43E-2  &2.79\\
$ $       &$256$       & 0.90E-3  & 2.88    & 0.90E-3  & 2.88  & 0.90E-3   & 2.88  & 0.90E-3   & 2.88  &0.90E-3  &2.88\\
$ $       &$512$       & 0.17E-3  & 2.92    & 0.17E-3  & 2.92  & 0.17E-3   & 2.92  & 0.17E-3   & 2.92 &0.17E-3  &2.92\\
$ $       &$1024$      & 0.31E-4  & ---    & 0.31E-4  & ---  & 0.31E-4   & ---  & 0.31E-4   & ---  &0.31E-4  &---\\

$3$       & $32$       & 0.15E-1  & 3.27    & 0.15E-1  & 3.27  & 0.15E-1   & 3.27  & 0.15E-1   & 3.27 
& 0.15E-1   &3.27\\
$ $       &$64$        & 0.28E-2  & 3.55    & 0.28E-2  & 3.55  & 0.28E-2   & 3.55  & 0.28E-2   & 3.55
& 0.28E-2   & 3.55\\
$ $       &$128$       & 0.41E-3  & 3.73    & 0.41E-3  & 3.73  & 0.41E-3   & 3.73  & 0.41E-3   & 3.73 &0.41E-3  & 3.73\\
$ $       &$256$       & 0.51E-4  & 3.84    & 0.51E-4  & 3.84  & 0.51E-4   & 3.84  & 0.51E-4   & 3.84   &0.51E-4  & 3.84\\
$ $       &$512$       & 0.55E-5  & ---    & 0.55E-5  & ---  & 0.55E-5   & ---  & 0.55E-5   &---  &0.55E-5  &---\\
%$ $       &$1024$      & ---  & ---    & ---  & ---  & ---   & ---  & ---   & ---  &---  &---\\

\cline{1-12}
\end{tabular*}
\label{table:2}
\end{table}

\section{Statements and Declarations}
\subsection{Funding}
This research is supported by National Natural Science Foundation of China (11771257) and Shandong Provincial Natural Science Foundation, China (ZR2023YQ002).
\subsection{Data availability statement}
The authors confirm that the data supporting the findings of this study are available within the article and its supplementary materials.
\subsection{Conflict of interests}
The authors declare that they have no conflict of interest.

%%%%%%%%%%%%%%%%%%%%%%%%%%%%%%%%%%%%%%%%%%%%%%
%
%
%
%%%%%%%%%%%%%%%%%%%%%%%%%%%%%%%%%%%%%%%%%%%%%%
%\cite{Ma1Zha2:2023-S}
%%\bibliographystyle{spmpsci}      % mathematics and physical sciences
%\bibliography{xiao}       % APS-like style for physics
%\bibliography{\myref/method/FE-book,\myref/method/FE-paper,\myref/method/FD-paper,\myref/problem/LM-paper,\myref/method/alex-paper,\myref/problem/singular-perturbation-problem}


\end{document}